\documentclass[12pt]{amsart}
\usepackage{amsthm, amsfonts, amssymb, amsmath}
\usepackage[letterpaper,margin=1in]{geometry}
\usepackage[usenames,dvipsnames]{pstricks}
\usepackage{epsfig}
\usepackage{graphicx}

\begin{document}
\newcommand{\norm}[1]{\left\| #1 \right\|}
\newcommand{\set}[1]{\left\{#1\right\}}
\newcommand{\R}{\mathbb{R}}
\newcommand{\fl}[1]{\lfloor #1 \rfloor}
\newcommand{\ceil}[1]{\lceil #1 \rceil}
\newcommand{\n}{\mathbb{N}}
\newcommand{\ds}{\displaystyle}
\newcommand{\z}{\mathbb{Z}}
\newcommand{\mb}[1]{\mathbb{#1}}
\newcommand{\T}{\mathbb{T}}
\newcommand{\I}{\mb{I}}
\renewcommand{\le}{\leq}
\renewcommand{\ge}{\geq}
\newcommand{\del}{\delta}
\newcommand{\m}{\mu}
\newcommand{\al}{\alpha}
\newcommand{\emf}[1]{\emph{#1}}
\newcommand{\ld}{\lambda}
\newcommand{\tx}[1]{\text{#1}}
\newcommand{\cl}[1]{\overline{#1}}
\newcommand{\tpl}[1]{(#1_1,\dots,#1_n)}
\newcommand{\tplt}[2]{(#1_1 #2_1,...,#1_n #2_n)}
\newcommand{\con}{\subset}
\newcommand{\s}{\mb{S}}
\newcommand{\ub}[1]{\underbrace{#1}}
\newcommand{\e}{\epsilon}
\renewcommand{\k}{\kappa}
\newcommand{\tab}{\hspace*{1em}}
\newcommand{\fr}[2]{\frac{#1}{#2}}
\renewcommand{\ds}{\displaystyle}
\newcommand{\f}[1]{\fr{1}{#1}}
\newcommand{\eq}{equation}
\newcommand{\D}{\Delta}
\newcommand{\Q}{\mb{Q}}

\newtheorem{theorem}[equation]{Theorem}
\newtheorem{definition}[equation]{Definition}
\newtheorem{lemma}[equation]{Lemma}
\newtheorem{proposition}[equation]{Proposition}
\newtheorem{corollary}[equation]{Corollary}
\newtheorem{conj}[equation]{Conjecture}

\theoremstyle{definition}
\newtheorem{example}{Example}

\title{The Lonely Runner Conjecture}
\thanks{I would like to acknowledge Professor Lind for the insight and support he gave freely during this project, and my family for supporting everything I do.}
\author{Clayton L. Barnes, II}
\maketitle
\pagebreak

%
\section{Origins}
Imagine two runners traveling, at distinct integer speeds, in the same direction along a circular track of unit perimeter.  Provided that the runners continue at their respective speeds indefinitely, it is clear that there will be a time when the runners are directly across from each other. More precisely, there is a time when the locations of the runners divide the circular track into two equal portions. Alternatively, there is a time for each $i$ when no other runner is within a half-track length of the runner $r_i$. Here the distance is to be taken along the track. Notice that because we are considering only two runners, if  $r_1$ is lonely at a given time then $r_2$ will also be lonely at that time. Moreover, loneliness will occur no matter the speeds of the runners so long as they are distinct. We can generalize this and consider three runners, with distinct integer speeds, running on the same unit track. However, if the use of ``lonely" remains the same as the case of two runners, one can give three runners speeds in such a way so no runner will be lonely. Giving the runners respective speeds of 1, $2$, and 3 laps per time unit will not admit any lonely runners. In order to allow the possibility, we alter the definition of ``lonely" with respect to the number of runners: 

\begin{definition}
If there are $k$ runners on the track with distinct speeds, a runner $r_i$ becomes $\emf{lonely}$ at some given time if none of the other $k-1$ runners are within a distance of $1/k$ of $r_i$ at that time.
\end{definition}

As in the beginning example, the distance is taken along the track of circumference 1. The lonely runner conjecture is the following:

\begin{conj}
Let $k$ be an arbitrary natural number, and consider $k$ runners with distinct, fixed, integer speeds traveling along a circle of unit circumference. Then each runner becomes lonely at some time.
\end{conj}

The problem of the lonely runner is interesting for several reasons. First the conjecture is relatively intuitive to grasp and easy to state. Most any mathematician, or any person for that matter, can understand the problem statement in little time. Secondly, the lonely runner conjecture (LRC) has equivalent statements that are seemingly unrelated at first glance. We will survey these equivalent formulations in the order of their discovery, proving their equivalence and the interesting results relating them. In the end, we will see that these equivalent formulations of the LRC are quite similar, and that each contributes a new perspective on the overall problem. However, what really makes the LRC interesting, is that after more than fifty years since its discovery, it remains unsolved. Currently it is known to hold for up to and including seven runners. The difficulty of proving the LRC may at first seem to increase exponentially with the number of runners. For two runners the problem is trivial; three runners is no more than a page to prove; the first proof of six runners was almost fifty pages, by a group of mathematicians from MIT [3]. But, a more clever argument by the French mathematician Jerome Renault proved the case of six runners in nine pages [7]. More recently two Spanish mathematicians prove the case for seven runners [2].
\\
$\tab$ For a real number $x$, let $\norm{x}$ denote the distance from $x$ to the nearest integer. If we are working in more than one real dimension then $\norm{x}$ is the distance from the vector $x$ to the closest integer lattice point. We must make the following notational note: We denote $\n$ as the natural numbers excluding 0, while $\n_0$ includes 0. Let $k \in \n$ and consider $k$ runners with distinct speeds $s_1, ..., s_k$ which are all in $\n_0$.  A runner with speed $s_i$ is lonely at time $t$ if and only if $\ds \norm{(s_j - s_i)t} \geq 1/k$ for all $j \neq i$. This shows that the LRC is equivalent to the following:

\begin{conj}
For each $k \in \n$ define the set $S_k = \set{s \in \n_0^k: \,s = (s_1, ... ,s_k), s_i \neq s_j  \emph{ for } i \neq j}$. If $s = (s_1, ..., s_k) \in S_k$, for each $i$ there is a $t \in \R$ where 
$\norm{(s_i - s_j)t} \geq 1/k$ for all $j \neq i$.
\end{conj}

Define the function $\del_k: \n^k \to \R$ by $\del_k(s) = $sup$_{t \in \R} $min$_{1 \leq i \leq k} \norm{s_i t}$. Then,

\begin{proposition}
The lonely runner conjecture is equivalent to $\ds\inf_{s \in \n^k} \del_k (s) \geq 1/(k+1)$ for all $k \in \n$.
\end{proposition}

Before proving this, we make a few observations. First, the function $\del_k$ can be thought of as a real function whose domain consists of $k$ runners with nonzero speeds. It is not required that these speeds be distinct. If we do make such a requirement and restrict $\del_k$ to $C_k= S_k \cap \n^k$, then assuming $\del_k|_{C_k} (s) \geq 1/(k+1)$ for all $s \in C_k$, it easily follows that $\del_k (s) \geq 1/(k+1)$ for all $s \in \n^k$.

\begin{proof}[Proof of Proposition 4]
Let $k \in \n$. Assume the lonely runner conjecture holds, and choose $(s_1, ... ,s_k) \in C^k$. Then $(0, s_1, ..., s_k) \in \n_0^{k+1}$ and all $s_i$ are all nonzero, so there is a time $t$ when the runner with zero speed is lonely, i.e., $\norm{s_i t} = \norm{(s_i -0)t} \geq 1/(k+1)$ for all $1 \leq i \leq k$. Thus sup$_{t \in \R}$min$_{1 \le i \le k} \norm{s_i t} \geq 1/(k+1)$, so that inf$_{s \in \n^k} \del_k(s) \geq 1/(k+1)$.
	Assume $\del_k(n) \geq 1/(k+1)$ for all $n \in \n^k$, we show Conjecture 2 holds. Pick $(s_1,..., s_k) = s \in S_k$. For an arbitrary $s_i$ we show that the runner whose speed is $s_i$ becomes lonely. Since $s_i \neq s_j$ for $i \neq j$, we have $s' = (|s_1-s_i|,..,|s_{i-1}- s_i|, |s_{i+1} - s_i|,.., |s_k - s_i|) \in \n^{k-1}$ so that by hypothesis, $\del_{k-1} (s') \geq 1/k$. Hence there is a time $t$ with $\norm{|s_j -s_i| t} = \norm{(s_j - s_i)t} \geq 1/k$ for $j \neq i$. Thus the runner is lonely at time $t$, i.e., Conjecture 2 holds.
\end{proof} 

One can glean from the proof above that the LRC holds for $k$ runners if and only if the infimum condition on $\del_n$ holds for $n = k-1$. Also, if one extends $\del_k$ in the obvious way to a function $\del_k'$ on $\z^k$, we have $\ds\inf_{s \in \z^k} \del_k'(s) = \ds \inf_{s \in \n^k} \del_k(s)$ since $\norm{st} = \norm{-st}$. What this says in terms of the LRC is that the direction of the runners is irrelevant: a general case of the LRC where one considers runners of both directions is implied by the originally stated  LRC where runners travel counterclockwise.
	 Although the LRC seems like a natural question following from the first example of two runners, the problem was not originally asked in this way. It was first posed as a problem in diophantine approximation relating the function $\del_k$ to orbits of irrational $k$-tuples in the $k$-dimensional torus, and few years later as an equivalent problem in view-obstruction. One may wonder about the LRC in a context where the runners' speeds are arbitrary. The proposition below shows that the LRC where the runners' speeds are natural numbers implies the more general case where the runners have arbitrary real speeds.

\begin{proposition}
Let $k$ runners have arbitrary distinct real speeds $\set{r_i}_1^k$. If the LRC holds, then each runner becomes lonely.
\end{proposition}
\begin{proof}
We use Lemma 6 in the next section. Let $1 \le i \le k$, we show the runner with speed $r_i$ becomes lonely. Set $s_j = r_j - r_i$ for $j \neq i$ and re-index as to consider the set $\set{s_l}_1^{m}$. If the $s_j$ are rationally independent then this is trivial, as the orbit of $(s_1,\dots,s_m)$ is dense in the $m$-torus. Thus assume the maximum number of rationally independent $s_i$ is $1< v < m$. We may also assume that all the $s_i$ are irrational by considering $\set{\beta s_i}_1^m$ for an appropriate irrational $\beta$. By Lemma 6 there is an irrational $\alpha$ and integers $w_i$, not all of which are 0, where $\cl{O(\alpha')} \con \cl{O(s_1,\dots,s_{m})}$ such that $\alpha' = \alpha(w_1,\dots, w_m)$ and $w_i=c$ for $v$ number of $i$ for some nonzero integer $c$. Here $O(\alpha)$ is the orbit of $\alpha$ in the $m$-torus. The LRC implies that $\norm{t'w_i} \ge 1/(m-v+2)$ for some $t'$ and all $i$, by assumption that the $w_i$ are integers. Hence there is a time $t = t'/\alpha$ when $\norm{\alpha t' w_i} > 1/(m+1)$ for each $i$. Since $\alpha(tw_1,\dots,tw_m) \in t (\cl{O(s_1,\dots,s_m)})$, and as $1/(m-v+2) > 1/(m+1)$, we have that $t (\cl{O(s_1,\dots, s_m)}) \cap \T^m\setminus[\f{m+1}, \fr{m}{m+1}]^m \neq \emptyset$. Thus $tO(s_1,\dots,s_m) \cap \T^m\setminus[\f{m+1}, \fr{m}{m+1}]^m \neq \emptyset$ as the latter set is open, so there is an integer $q$ when $\norm{qts_i} \ge 1/(m+1)$ for each $i$. Hence there is a time $qt$ when $\norm{qt(r_j - r_i)} \ge 1/(m+1) = 1/k$ for each $j \neq i$.
\end{proof}

%
\section{Diophantine discovery}
In 1967, the German mathematician J\"{o}rg Wills wrote an article containing five related topics on diophantine approximation. The last two topics relate closely to the LRC via the function $\del_k$. Wills was the first to explicitly find bounds for $\del_k$. Following in his exposition, we first define the following functions. For $n \in \n$, let $\mu : \z \times \R^n \to \R$  be $ \m(q, a) = \text{min}_{1 \leq i \leq n} \norm{q a_i}$ where $a = (a_1, ..., a_n)$, and 
$$\ld_n(a) = \ds\sup_{q \in \z} \mu (q, a).$$
\begin{example}
If $n =2$ so $a = (a_1, a_2)$, to picture $\ld_2(a)$ first imagine the additive group generated by $a$ as a subset $S/ \z^2$ of the two torus $\T^2$, where $S = \set{ qa : q \in \z}$. If $c = \ld_2(a)$, any point $(x, y) \in S$ has either $\norm{x} \le c$ or $\norm{y} \le c$. That is, $1-2c$ is largest length of any square centered at $(\f{2}, \f{2})$ whose interior does not intersect $S/ \z^2$. For if $(x,y)$ were in the interior, $\norm{x} > \f{2} - \fr{1-2c}{2} = c$ and the same for $y$. Thus $S/\z^2$, which is the orbit of $a$ in $\T^2$, intersects any square centered at $(\f{2}, \f{2})$ whose length is greater than $1-2c$, and $c$ is the smallest number for which this holds.
\end{example}
In his paper, Wills considers the case when $n$ is a nonzero natural number, and $a$ is an irrational $n$-tuple. We let $\I$  the set of irrationals, and prove a lemma before stating the main result relating the functions $\del_n$ and $\ld_n$. First, some notation.
For $a = (a_1, ... ,a_n) \in \R^n$, define 
$$Q(a) = \set{(s_1, ...,s_n) \in \R^n : s_i = q a_i + t_i, \, q, t_i \in \z , \, i=1,..,n}$$ 
In the case when $n = 2$ reconsider the set $S$ above. The elements in the corresponding set $Q(a)$ are real ordered pairs, $(s_1, s_2)$, whose difference with a pair of integers $(t_1, t_2)$ is in $S$. Here we take addition between pairs to be coordinate wise. In other words an element $x$ is in $Q(a)$ if and only if  $x=t$\,mod$(S)$ for some $t \in \z^2$.
\begin{lemma}
For $n \in \n$, $a = \tpl{a} \in \I^n$, consider the set $Q(a)$. There exists an irrational number $\alpha$ and $s_i \in \z$, such that $\cl{Q(a')} \con \cl{Q(a)}$ where $a' =(s_1\alpha,...,s_n\alpha)$. If m is the dimension of a rationally independent basis for $\set{a_i}$, then $s_i=c$ for an m number of i, where c is a nonzero integer.
\end{lemma}

\begin{proof}

Let the $n$-tuple of irrationals $a = \tpl{a}$, be given. Now if $s$ is in the closure of $Q(a)$ then $s$ mod$\,\z^n$ is in the closure of $Q(a)$ equipped with the norm $\norm{ \, \cdot \,}$. If $s$ is in the closure of $Q(a)$ with respect to this norm, then $s+t \in \cl{Q(a)}$ for some $t \in \z^n$. But then $s \in \cl{Q(a)}$ because $Q(a)$ is closed under addition with $\z^n$. Thus to show that there is some irrational $\alpha$ and integers $s_i$ with $a' = (s_1 \alpha,\dots, s_n \alpha)$ having the desired property, it suffices to show the existence of such an $\alpha$ and $k_i$ in $Q(a)$ under $\norm{\, \cdot \,}$, which is the closure of $O(a)=\set{qa}_{q \in \z}$ in the $n$-torus: we denote the closure of a set $A \con \T^n$ as $\cl{A}$ in the rest of the proof.
If all of the $a_i$ are independent over the rationals, then $O(a)$ is dense in the $n$-torus, so there is nothing to prove. 
Otherwise assume that $\set{a_i}_1^n$ are not all rationally independent, i.e., some of the $a_i$ are rationally dependent. Then there exists a rationally independent subset $\set{b_i}_1^m \con \set{a_i}_1^n$ so that for each $i$, $a_i = \sum_{k=1}^m r_k^i b_k$ for some rational numbers $r_k^i$. Since the $b_i$'s are rationally independent, setting $b = (b_1,\dots,b_m)$ results in $O(b)$ being dense in the $m$-torus. Hence for any irrational $\alpha \in \T$ there is a sequence of integers $\set{q_k}_1^\infty$ such that for each $b_i$, $\norm{q_kb_i - \alpha} \rightarrow 0$. Then we have $\norm{r_k^i(q_nb_k - \alpha)} \rightarrow 0$ as $n \to \infty$ for each $k$, $i$ and also have $\norm{q_na_i - \sum_{k=1}^m r_k^i\alpha} = \norm{q_n \sum_{k=1}^m r_k^i b_k - \sum_{k=1}^m r_k^i \alpha} \rightarrow 0$ as $n \to \infty$. Thus $w=(p_1\alpha,\dots,p_n\alpha)$ is in the closure of $\cl{O(a)}$ in $\T^n$, where $p_j = \sum_{k=1}^m r_k^j$ if $a_j \notin \set{b_i}_1^m$, and $p_j = 1$ otherwise. Therefore $\cl{O(w)} \con \cl{O(a)}$ and by multiplying the $p_n$ by an appropriate integer $c$, all $cp_i$ will be integers and setting $a' = cw$ proves the result. 
\end{proof}

We use the above lemma to prove the following relationship between $\del_k$ and $\ld_k$. \\

\begin{theorem}
For any $n \in \n$ and $\e \geq 0$, the following are equivalent:

\begin{enumerate}
\item There is an $s = \tpl{s} \in \n^n$ with $\del_n(s)$ = $\emf{sup}_{t \in \R} \emf{ min}_{1 \leq i \leq n} \norm{s_i t} \le \e$. \\
\item There is an $a = \tpl{a} \in \I^n$ with $\ld_n(a)$ = $\emf{sup}_{q \in \z} \emf{ min}_{1 \leq i\leq n} \norm{q a_i} \le \e$.
\end{enumerate}

That is, $\emf{inf}_{a \in \I^n} \ld_n(a)$ = $\emf{inf}_{s \in \n^n} \del_n(s)$.
\end{theorem}

\begin{proof}
(1) $\implies$ (2).
Assume (1) holds, that is, $\del_n(s) \leq \e$. Then for $\alpha \in \I$ and an integer $q$, set $t = q \alpha$ and $s_i \alpha = a_i$ for $i=1,\dots,n$. Then $a = \tpl{a} \in \I^n$ and 
$$ \tx{min}_{1 \leq i \leq n} \norm{q a_i} = \tx{min}_{1 \leq i \leq n} \norm{s_i t} \leq \e.$$
As $q$ was an arbitrary integer, $(2)$ follows.

(2) $\implies$ (1).
Assume (2) holds above, with $\ld_n(a) \leq \e$. Consider the set of reals 
$$S(\e) = \set{\tpl{t} \in \R^n: \tx{min}_{1 \le i \le n} \norm{t_i} \, \le \, \e}$$
First notice that $\cl{Q(a)} \con S(\e)$ as min$_i \norm{q a_i -t_i}$ = min$_i \norm{q a_i} \leq \e$ for any $q, t_i \in \z$, and because $S(\e)$ is closed. So by the lemma there is an irrational $\alpha$ and 
integers $s_i$ such that setting $a' = (s_1 \alpha,\dots,s_n \alpha)$ yields $\cl{Q(a')} \con \cl{Q(a)} \con S(\e)$. Now by definition 
$$Q(a') = \set{\tpl{x}: x_i = q \, s_i \, \alpha - k_i, \, q, \, k_i \in \z}$$
so by Kronecker's approximation theorem, for each $t \in [0,1)$, there is a sequence of integers $\set{q_k}_{k=1}^ \infty$ such that $\norm{q_k \alpha - t} \to 0$ as $k \to \infty$. That is, $q_k \alpha \to t$ in $\norm{\, \cdot \,}$ so that $s_i \, q_k \, \alpha \to s_i t$ under $\norm{\, \cdot \,}$.  Since $q_k \, \alpha$ mod(1) is in $Q(a)$, it follows that $(s_1 t,\dots,s_n t) \in \cl{Q(a')}$. It thus follows that the set $\set{\tpl{x} : x_i = s_i t - k_i, \, t \in \R, \, k_i \in \z} \con \cl{Q(a')} \con S(\e)$. We show that $s = (|s_1|,\dots,|s_n|) \in \n^n$ is the required element with $\del_n(s) \le \e$: But this is immediate as for any $t \in \R$ we have $\tpl{ts} \in S(\e)$, hence $\e \ge$ min$_{1 \le i \le n} \norm{s_i t}$ = min$_{1 \le i \le n} \norm{ |s_i| t}$. Thus $\del_k(s) \le \e$ and (1) holds.
\end{proof}

For $n \in \n$, set $\kappa(n) = \ds\inf_{a \in \I^n} \ld_n(a)$, then 

\begin{proposition}
The lonely runner conjecture is equivalent to $\k(n) \ge 1/(n+1)$ for every $n \in \n$.
\end{proposition}
This follows immediately from Theorem 7 and Proposition 4.
\\
The first and original conjecture, which is equivalent to the LRC, was that $\kappa(n) \ge 1/(n+1)$ for every positive natural number $n$. The first attempts at a general solution were by method of establishing sufficient bounds for $\kappa$. When making these attempts it is has been profitable to exploit Theorem 7 and find bounds for the function $\del_k$. Since the definition of $\del_k$ is more discrete, it is easier to apply simple combinatorial results or methods that give equivalent results for $\kappa$ that are otherwise non-obvious. Even the pigeon hole principle provides us with a sharp upper bound for $\kappa$, that is

\begin{proposition}
For any natural number n, $\kappa(n) \leq 1/(n+1)$.
\end{proposition}

This result has a couple of interesting consequences to the LRC. For one, it tells us that the our definition of ``loneliness", as stated in Definition 1, cannot be relaxed any further if there is to be any hope of the LRC holding. Secondly, the proof of the above proposition provides the speeds of the runners for which the loneliness condition is tight. Consider the example with three runners with speeds of 1, 2 and 3 units per second. We stated without proof that without a definition of loneliness that accounts for the number of runners, the first runner is never lonely. Here we prove a corresponding result with $n$ runners. That is, if there are $n$ runners on the track with speeds of $1,2,\dots,n$, then there is not a time when all the other runners are further than $1/n$ from the first runner. We now prove the proposition using the pigeon hole principle, or more precisely Dirichlet's box principle.

\begin{proof}
We use the fact that $\kappa(n) = \ds\inf_{s \in \n^n} \del_n(s)$ for every $n \in \n$, according to Theorem 7. That is, we show

$$\del_n(1,2,\dots,n) = \ds\sup_{t \in \, \R} \, \ds\min_{1 \le i \le n} \norm{i t} = \ds\sup_{t \, \in \, [0,1]} \ds\min_{1 \le i \le n} \norm{i t}  = \f{n+1}.$$

We show that the supremum occurs at $t_0 = \f{n+1}$. Notice that

$$\displaystyle\min_{1 \le i \le n} \norm{i t_0} =  \tx{min} \set{\norm{1/(n+1)},\norm{2/(n+1)},...,\norm{n/(n+1)}} = \ds\f{n+1}.$$

It also follows that for any $t \in [0, 1/(n+1)$ we have $\min_{1 \le i \le n} \norm {i t} \le \norm{1 t} < \f{n+1}$.
Thus we assume for contradiction that there is a time $t$ when $\norm{i \, t} > 1/(n+1)$ for all $1 \le i \le n$. It follows that $t \in (\f{n+1}, 1]$. Since $t$ is obviously not $1$, the set of points $\set{\norm{i t}}_{i=1}^n = \set{a_i}_{i=1}^n$ form a partition of the unit interval. By assumption no $a_i$ is within $1/(n+1)$ of 0 or 1, so that $\set{a_i}_1^n \con (\f{n+1}, \fr{n}{n+1}) = I$. These $n$ points in $I$ form $n-1$ closed ``boxes", i.e., intervals, $b_1,\dots,b_{n-1}$ where $b_1$ = $[a_{i_0}, a_{i_1}]$ in which $a_{i_0}$ is the closest element in $\set{a_i}_1^n$ to 0, $a_{i_1}$ is the next greater,...etc. It is tempting to think that letting the boxes have the form $[a_{i}, a_{i+1}], \, 1 \le i \le n,$ would suffice; but this will not generally work as some boxes may overlap: this is because we are considering $a_i$=$i \, t$ mod1, so some of the $a_i$ will wrap around the unit interval, possibly altering their inherent order. Thus defined, a given box $b_k$ has length 
$\norm{a_{i_k} - a_{i_{k+1}}} = \norm{i t - j t}$ for some $1 \le i < j \le n$. But this is $\norm{(i - j)t}$ and $i-j$ is a number in between 1 and $n$, so by assumption, $\norm{(i-j)t} > \f{n+1}$. 
Thus each box $b_k$ has length greater than $1/(n+1)$, and since there are $n-1$ boxes, their total length is strictly greater than $(n-1)\, \f{n+1} = \fr{n-1}{n+1}$. But these are inside $I$, which has length $1 - \fr{2}{n+1} = \fr{n-1}{n+1}$, and this is impossible as the total length of the disjoint (excepting their boundaries) boxes $b_1,\dots,b_{n-1}$ would be greater than a box $I$ containing them. Thus no such time $t$ exists, and the proof is complete.
\end{proof}

In light of Proposition 4, there are really $n+1$ runners, with the first having 0 speed. This above proposition says that for any $n$ runners with speeds of $\set{1+a, 2+a, 3+a,..., n+a}$, the
first runner with speed of $1+a$ is never separated by more than $1/n$ from the other runners. We now prove the LRC for three runners

\begin{theorem}
The lonely runner conjecture holds for 3 runners, that is, we have $\kappa(2) = \ds\inf_{s \in \n^2} \del_2(s) = \f{3}$.
\end{theorem}
\begin{proof}
We show that $\del_1(s) \ge 1/3$ for every $s \in \n^2$. Assume for sake of contradiction that there is a $k = (k_1, k_2) \in \n^2$ where $\del_2(k) = \ds\sup_{t \in \R} \min \set{\norm{k_1 t}, \norm{k_2 t}} < 1/3$.
Without loss of generality say $k_1 \le k_2$. Set $t_1 = \f{3k_1}$ so that $0 \le t_1 \le 1$ as $k_1 \ge 1$, and we have $\norm{k_1 t} = 1/3$. Then by assumption,
min$\set{\norm{k_1 t_1}, \norm{k_2t_1}} < 1/3$ so we have $\norm{k_2 t_1} \le \alpha < 1/3$. Now $k_1 \le k_2$ implies $k_2t_1=\fr{k_2}{3k_1} \ge \f{3} > \alpha$ so that
$\norm{k_2/(3k_1)} \ge 1-\alpha > 2/3$ following from the definition of the norm. Then $\fr{k_2}{k_1} > 2,$ so there exists a natural number $g$ with 

$$\fr{k_2}{k_1} \le g < g+1 \le \fr{2k_2}{k_1},$$
and multiplying both sides by $\fr{k_1}{k_2}$ yields

\begin{equation} \label{clay}
1 \le \fr{k_1}{k_2}g < \fr{k_1}{k_2}(g+1) \le 2.
\end{equation}

It follows from the fact that $g$ is a natural number, that either $g$ or $g+1$ is not divisible by 3. Select $g' \in \set{g, g+1}$ so that 3 $\nmid g'$, and set $$t_2 = \fr{g'}{3k_2},$$ by dividing $\eqref{clay}$ by 
$3k_1$, and by dividing $\eqref{clay}$ by 3, we have the inequalities
$$ 0 \le t_2 \le 1, \, 1/3 \le k_1 t_2 \le 2/3.$$ 
It follows from the above that $\norm{k_1t_2} \ge 1/3$. We also have $k_2t_2 = g'/3$, and since $g'$ is not divisible by 3, it easily follows that $\norm{k_2 t_2} = \norm{g' /3} = 1/3$. But this contradicts the hypothesis that $\del_2(k) = \ds\sup_{t \in \R} \, \min \set{\norm{k_1t}, \norm{k_2 t}} \le \alpha < 1/3$. Hence, there is no such $k \in \n^2$ and so by Theorem 7 and Proposition 9, $\kappa(2) = 1/3$. By Proposition 4, the lonely runner conjecture holds for three runners.
\end{proof}

Since the LRC is known only up to and including 7 runners, $\kappa(n)$ is known to be $1/(n+1)$ for $n$ up to and including 6. Yet we do have the following bounds on $\kappa(n)$ for any $n \in \n$.

\begin{proposition}
For all $n \in \n$, $\f{2n} \le \kappa(n) \le \f{n+1}$.
\end{proposition}

\begin{proof}
By Proposition 8, we have $\kappa(n) \le \f{n+1}$. We show that $\del_n(k) \ge \f{2n}$ for any given $k = \tpl{k} \in \n^n$. For $\e \in [0, 1/2]$ and $s \in \n$ we have $$\norm{st} \le \e \tx{ when } 0 \le t \le \fr{\e}{s},$$ and $$\norm{st} > \e, \tx{ for } \fr{\e}{s} < t < \fr{1 - \e}{s}.$$
It follows that the interval with $t \in [0,1/s]$ having $\norm{st} \le \e$ has a length of $\fr{2\e}{s}$. Since $\norm{st} = \norm{s(t+\f{s})},$ there are $s$ such intervals in [0,1]. Call the union of these intervals $I$. Then $I \con [0,1]$ has length of $2\e$, and every $t \in I$ has $\norm{st} \le \e$. 

Now let $k=\tpl{k} \in \n^n$ be arbitrary. For each $1\le i \le n$, define 
$$J_i(k) = \set{t \in [0,1]: \norm{k_it} \le \del_n(k)}.$$
Thus the length of each $J_i(k)$ is $2\del_n(k)$. Also, each $t \in [0,1]$ must belong to some $J_i(k)$, since $\ds\min_{1 \le i \le n} \norm{k_i t} \le \ds\sup_{t \in \R} \ds\min_{1 \le i \le n} \norm{k_i t} = \del_n(k)$. Hence there is some $i$ with $\norm{k_it} \le \del_n(k)$. Thus, 

$$[0,1] \con \bigcup_{i=1}^n J_i(k),$$

so the length of [0,1] is less than or equal to the sum of the lengths of the $J_i(k)'s$. That is, 
$$1 \le 2n\del_n(k),$$ so
$$\f{2n} \le \del_n(k), \, \forall k \in \n^n.$$ Hence,

$$\f{2n} \le \ds\inf_{k \, \in \, \n^2} \del_n(k).$$ It then follows from theorem 7 that $\kappa(n) \ge \f{2n}$.
\end{proof}

Given $n$ runners on a unit track, this lower bound for $\kappa$ shows that eventually each runner will be sufficiently separated from the others.

\begin{proposition}
Let $n$ runners with distinct fixed integer speeds be traveling on a circle with unit circumference. For each runner there is a time when it is separated from every other runner by a distance of $\f{2(n-1)}$.
\end{proposition}

\begin{proof}
Let $n \in \n$, and $s = \tpl{s} \in \z^n$, where $s_j \neq s_i$ for $j \neq i$. Fix $i$, so that $s_i$ represents the speed of the $i$-th runner. Let $r_j = |s_j - s_i|$. Then $r_j \in \n$ when $i \neq j$, so that the number of $r_j$'s which are nonzero is $n-1$. By Proposition 12, 

$$\ds\sup_{t \, \in \, \R}\,\ds\min_{j \neq i} \norm{r_j t} \geq \f{2(n-1)}.$$

Hence there is a time $t$ when $\norm{r_j t} =\norm{|s_j-s_i|t} = \norm {(s_j-s_i)t} \geq \f{2(n-1)}$ for $j \neq i$. This proves the result.
\end{proof}

As previously mentioned, the LRC was originally formulated as the following:

\begin{conj}
For all $n \in \n$, $\kappa(n) = \f{n+1}$.
\end{conj}

%
\section{The View-Obstruction Problem}

In 1971, a few years after Wills' results were released, Thomas W. Cusick gave an equivalent reformulation of conjecture 13 as a conjecture in view-obstruction.
Let $E_n$ denote the region in $\R^n$ where all coordinates are positive, so any $x = \tpl{x} \in E_n$ has $0<x_i< \infty$ ($i=1,\, 2,\dots,n$). Suppose that $C$ 
is a closed convex body in $\R^n$ and which contains the origin as an interior point. For each $\alpha \ge 0$, define $\alpha C$ to be the set of all $\tpl{\alpha x},$ where $\tpl{x}$ is a point
in $C$; hence $\alpha C$ is the scale of $C$ with the magnification of $\alpha$. Define $C+\tpl{m}$ to be the translation of $C$ by the point $\tpl{m} \in \R^n$.
\\
$\tab \emf{Statement of problem.}$ Define the set of points $\D(C,\alpha)$ by 

$$\D(C, \alpha) = \set{\alpha C + (m_1 + \f{2},..., m_n + \f{2}): m_i \in \n, \, i=1,...,n}.$$
Find the constant $K(C)$ defined to be the lower bound of those numbers $\alpha$ for which every ray $r(t) = (a_1 t,..., a_n t)$ where $a_i > 0, \, t \in [0, \infty),$
intersects $\D(C, \alpha)$.
\\
\tab That is, the region $S_n$ is divided into $n$-dimensional cubes of side length 1 with vertices at the integer coordinates. The set $\D(C, \alpha)$ is the set of translates of
$\alpha C$ to the centers of these cubes. For a given $\alpha$, there may be a ray $r$ contained in $E_n$ which does not pass through this set $\D(C, \alpha)$. Then $K(C)$ is
the supremum of all $\alpha$ where there is such a ray. Alternatively, $K(C)$ is the infimum of all such $\alpha$ where $\D(C, \alpha)$ intersects every such ray.
The problem of interest will be when $C_n$ is taken to be the $n$-dimensional cube with unit side lengths, centered at the origin. In this case we define $K_n := K(C_n).$

\includegraphics[scale=.4]{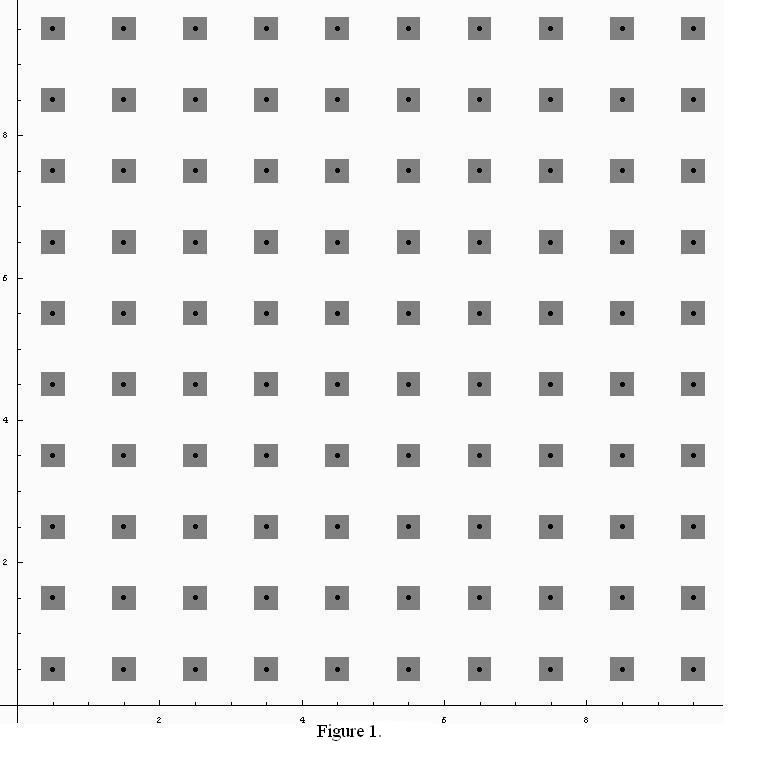} \includegraphics[scale=.45]{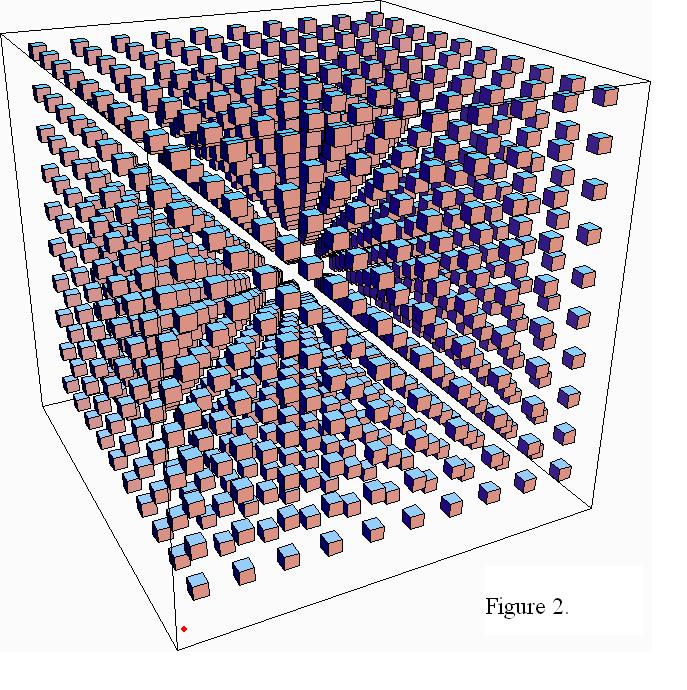}

Figure 1 on the left shows the set $\D(C_2, 1/3)$. The half-integer lattice points are depicted as the centers of the potentially view-obstructing squares. (From here on, the term ``half-integer" will refer to those numbers with nonzero integer part.) On the right, Figure 2 depicts the set $\D(C_3, 1/3)$; the dot in the lower front corner is the origin. We will see later that the squares in Figure 1 do obstruct all views and that the corresponding $\alpha = 1/3$  is the value $K_2$ that we seek. Likewise we will see that the cubes in Figure 2 fail to obstruct all views.

The main goal of the view-obstruction problem is to characterize the numbers $\alpha$ for $\D(C_n, \alpha)$ that obstruct all rays with direction $\tpl{r}$, $r_i$ is positive and real. We use the term ``direction" loosely, not requiring the vector to have unit length.
If one is to prove $\beta=K(C_n)$, it suffices to show (for any $\e > 0$)
\begin{enumerate}
\item $\D(C_n, \beta + \e)$ obstructs all views.
\\
\item $\D(C_n, \beta - \e)$ does not obstruct all views.
\end{enumerate}

Although the above conditions are not hard to grasp conceptually, it is not practical to tackle (1) and (2) as stated because ``all views" does not behave nicely.
In order to make important differences apparent, let $K_n'=K'(C_n)$ be the infimum of all $\alpha$'s such that $\D(C_n, \alpha)$ obstructs all rays with rational direction, i.e., rays that have the same direction as a rational $n$-tuple. It follows that $K_n' \le K_n$, since if $\D(C_n, \alpha)$ obstructs all views it necessarily obstructs all views with rational direction. Also, $K_i' \le K_n'$ for $i \le n$, as if a ray with rational direction $\tpl{r}$ is obstructed by $\D(C_n, \alpha)$, the ray with direction $(r_1,\dots,r_i)$ must necessarily be obstructed by $\D(C_i, \alpha)$. This would be less formidable if it suffices to prove (1) and (2) for all views of rational direction. This can be proved assuming $K_n' < K_m'$ when $n < m$.  The proof is reminiscent of Proposition 5.

\begin{proposition}
Assume $K_n' < K_m'$ when $n<m$. Then the set $\D(C_n, K_m')$ obstructs all views, that is, $K_m' = K_m$.
\end{proposition}

\begin{proof}
$\Rightarrow:$ If $\D(C_n, \alpha)$ obstructs all views then it trivially obstructs all views with rational direction.\\
$\Leftarrow$: Let $r: [0, \infty) \to S_n$ be the ray $r(t) = vt$, where $v=\tpl{r}$, $r_i$ positive and real. If the $r_i$ are rationally independent then there must be a time $t$ when $r(t)$ is sufficiently close to the set of half integer coordinates; as the orbit $O(v)= \set{qv}_{q\in \z}$ is dense in the $n$-torus.
Assume then that $\set{r_j}_1^n$ are not all rationally independent, i.e., some of the $r_j$ are rationally dependent. Since the ray $r(t)$ has the same direction as $r(\beta t)$, we can assume without loss of generality that the $r_j$ are irrational. There exists a largest rationally independent subset $\set{b_l}_1^m \con \set{r_j}_1^n$ for $m < n$. In the $n$-torus, the set $\D(C_n, \alpha)$ reduces to the single cube, denoted $G(\alpha)$, centered at $(\f{2}, \f{2})$ with length $\alpha$. The ray $r$ passes through $\D(C_n, \alpha)$ if and only if the image of $r$ in $\T^n$ intersects $G(\alpha)$. By Lemma 6 there is an irrational $w$ and integers $s_i$ such that $\cl{O(ws_1, \dots, ws_n)} \con \cl{O(r_1, \dots, r_n)} \con \T^n$. Also by Lemma 6, there is an $m$ number of $s_i$ equal to the same nonzero integer $c$. Let $v = n-m+1$. Since $h(t) = tu$, $u=(ws_1,\dots, ws_n)$, has rational direction, the hypothesis says that $h$ intersects $G(K_v') \con \tx{int}(G(K_m')) \con \tx{int}(G(K_m))$ in the $n$-torus. The closure of $h([0,\infty))$ is contained in the closure of $r([0,\infty))$. Since $\tx{int}(G(K_m))$ is open, the image of $r$ in the $n$-torus intersects $G(K_m)$. This proves the result.
\end{proof}

For our future purposes, in light of Conjecture 19 below, we can assume $K_n' = K_n$.

\begin{proposition}
For every n, $K'_n = 1-2 \ds\inf_{r \, \in \, \n^n} \del_n (r)= 1-2\kappa(n).$
\end{proposition}

\begin{proof}
We first show that 
$K'_n$ = 2 $\ds\sup_{r \in \n^n} \, \min_{t \in [0,1]} \, \max_{1 \le i \le n}$ $\norm{r_i t - \f{2}}.$
Let $s = \tpl{s} \in \n_0^n$ so that $s$ is the direction of the ray $r(t) = (s_1 t, ..., s_n t)$. The closest distance, in the product metric, from $\tx{Im}(r)$ to the nearest point of  $A=\set{(m_1 + \f{2},..., m_n + \f{2}): m_i \in \n}$ is $l=\ds\min_{t \in [0,1]} \, \max_{1 \le i \le n}$ $\norm{r_i t - \f{2}},$ as $\norm{r_it - \f{2}}$ is the distance from $r_i \, t$ to the nearest half-integer. It follows
that the smallest $\alpha$ for which $\D(C_n, \alpha)$ obstructs the ray $r(t)$ is 2$l$: since if $c=\norm{a-\f{2}}$, $a$ is contained in the interval $[b-\f{2} - c, b-\f{2} +c]$ for some integer $b$; the interval has length $2c$, and this is the smallest interval centered at the half-integers containing $a$. It follows that $K'_n$ is 2$\sup \ds\min_{t \in [0,1]} \, \max_{1 \le i \le n}$ $\norm{r_i t - \f{2}}$, where the supremum is taken over all $\tpl{r} \in \n^n$. One can check that
$\norm{r_i t} = \f{2} - \norm{r_it-\f{2}},$ and hence
$$K_n'=2 \ds\sup_{r \in \n^n} \, \min_{t \in [0,1]} \, \max_{1 \le i \le n} \tx{$\norm{r_i t - \f{2}}$}=2 \ds\sup_{r \in \n^n} \, \min_{t \in [0,1]} \, \max_{1 \le i \le n} (\f{2} - \norm{r_it}),$$
which becomes 
$$1 +2\ds\sup_{r \in \n^n} \, \min_{t \in [0,1]} \, \max_{1 \le i \le n} (-\norm{r_it}),$$
which is $$1-2\ds\inf_{r \in \n^n} \, \max_{t \in [0,1]} \, \min_{1 \le i \le n} \norm{r_it}.$$
But this is 
$$1-2\ds\inf_{r \, \in \n^n} \del_n(r)=1 - 2\kappa(n),$$
and the proof is complete.
\end{proof}

\begin{corollary}
Assuming $K_i' < K_j'$ when $i < j$, $K_n = 1-2\kappa(n)$ for all $n \in \n$.
\end{corollary}

\begin{proof}
Apply Proposition 15.
\end{proof}

\begin{corollary}
For any $n \in \n$, $\fr{n-1}{n+1} \le K_n' \le \fr{n-1}{n}$.
\end{corollary}
\begin{proof}

By Proposition 11, 
$$\f{2n} \le \kappa(n) \le \f{n+1}.$$
The rest follows by setting $K_n' = 1-2\kappa(n)$.
\end{proof}

Assuming the conjecture below, the above results ensure that finding $\kappa(n)$ and finding $K_n$ are equivalent problems. Since the LRC is equivalent to $\kappa(n) = 1/(n+1)$ for each $n$, Conjecture 14, and hence the LRC, is also equivalent to

\begin{conj}
For every $n \in \n, K_n' = \fr{n-1}{n+1}$.
\end{conj} 
As Proposition 5 shows the LRC with rational speeds implies the LRC where the runners speeds are arbitrary. Likewise by Proposition 15, if $\D(C_n, \fr{n-1}{n+1})$ obstructs all rational views for each $n$, then it obstructs all views. That is, if  $K_n' = \fr{n-1}{n+1}$ for each $n$, then $K_n=K_n'$. It is interesting to notice that by Corollary 17, $\D(C_n, \fr{n-1}{n})$ necessarily obstructs every view in $S_n$. From the lower bound it follows that $K_3 \ge 1/2 > 1/3$ so that the cubes represented in Figure 2 do not obstruct all views.

We proved the LRC with three runners in the previous section. Here we give an alternate proof using view-obstruction. 
\begin{theorem}
$K_2 = 1/3$.
\end{theorem}

\begin{proof}

We refer to the following picture in the proof:

\includegraphics[scale=.55]{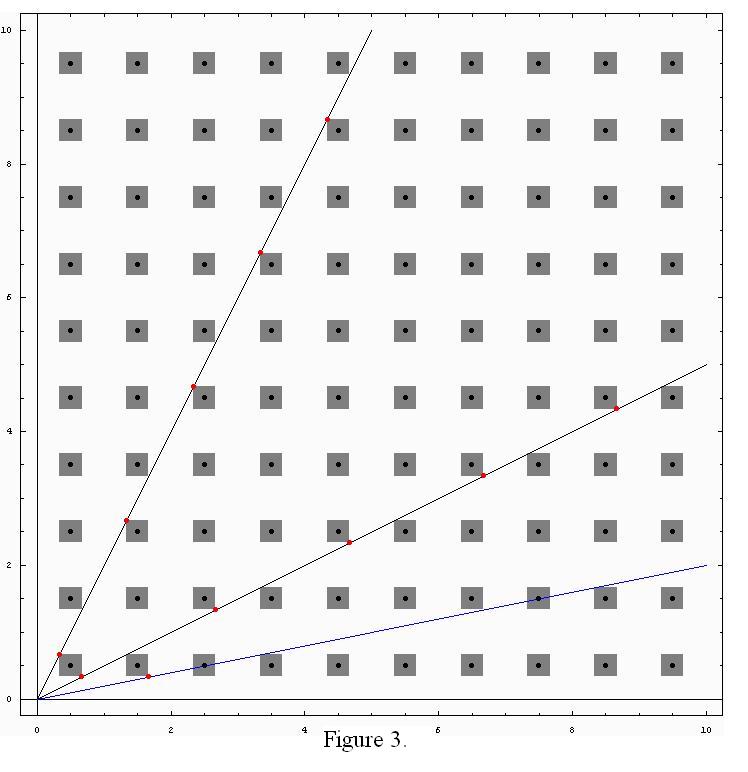}

The squares are centered at half-integers and have side length of $1/3$, as in figure (2), so the squares represent the set $\D_2=\D(C_2, 1/3)$. The two top rays have slopes of 2 and $1/2$ respectively while the bottom ray has a slope of $1/5$. By symmetry, $\D_2$ obstructs all views with slope $\theta > 2$ if it obstructs all slopes with $\theta \in (0, 1/2)$. The rays $y=\theta x$ with $1/2 \le \theta \le 2$ intersect the square with center $(\f{2}, \f{2})$. By observation of the bottom ray, every ray with slope $\theta \in (o, \f{2})$ are obstructed by a square. This is evident if the slope is between that of the bottom line and the middle line of slope $1/2$, i.e., $\theta \in [\f{5}, \f{2}]$. If $0 < \theta < \f{5}$ then the ray $y = \theta x$ has no hope of passing through the gap of two consecutive squares with centers $(\f{2} + n, \f{2}), \, (\f{2} +n+1, \f{2})$ as the minimal slope of a line needed to pass unobstructed between such consecutive squares is $1/2$. Thus the set $\D_2$ obstructs all views. From the above figure it is also evident that $1/3$ is the smallest number $\alpha$ such that $\D(C_2, \alpha)$ obstructs all views: for the top lines with slopes 2 and $\f{2}$ only pass through the corners of squares in $\D_2$, as indicated by the points along those lines. Hence $K_2 = \f{3}$ and this proves the LRC for three runners.
\end{proof}

The following figures give heuristic examples demonstrating that $K_3 = 1/2$. In all figures, the dot in the center is the origin, so that one is placed at the origin and able to literally ``see" which views are obstructed by the cubes.

\includegraphics[scale=.3]{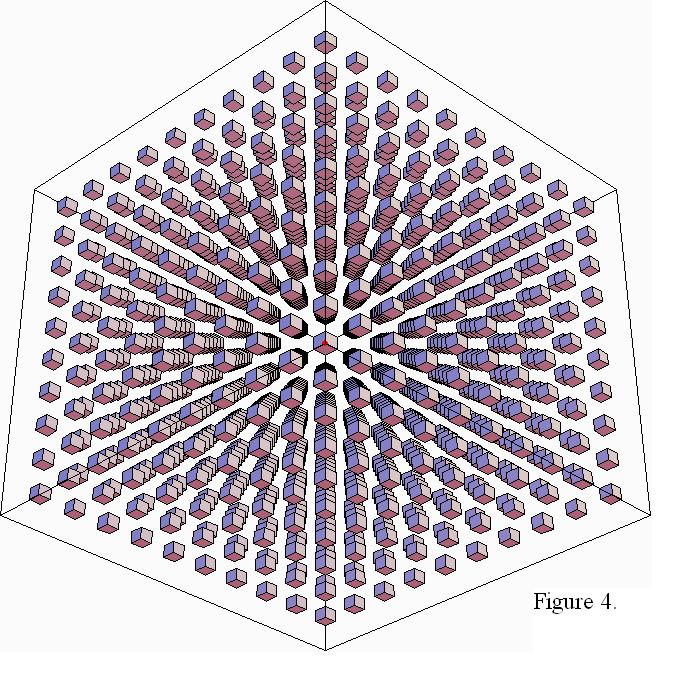} \includegraphics[scale=.3]{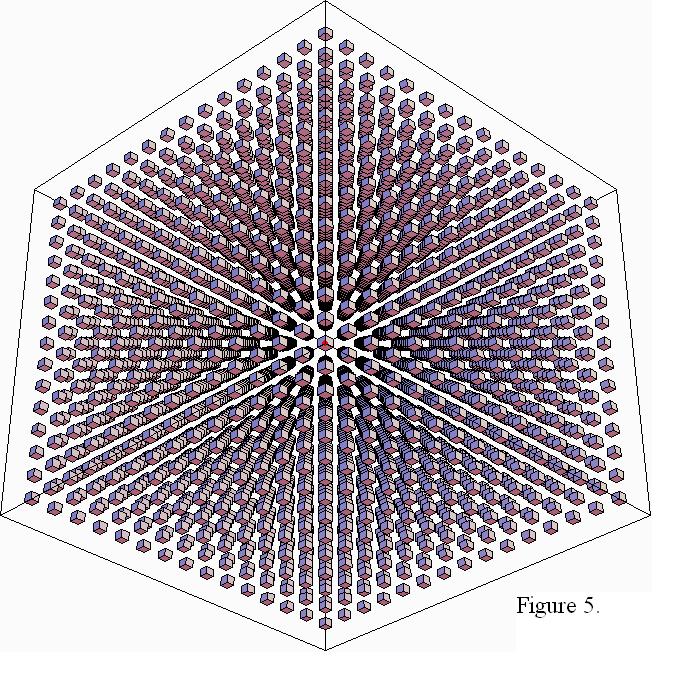} 
\includegraphics[scale=.3]{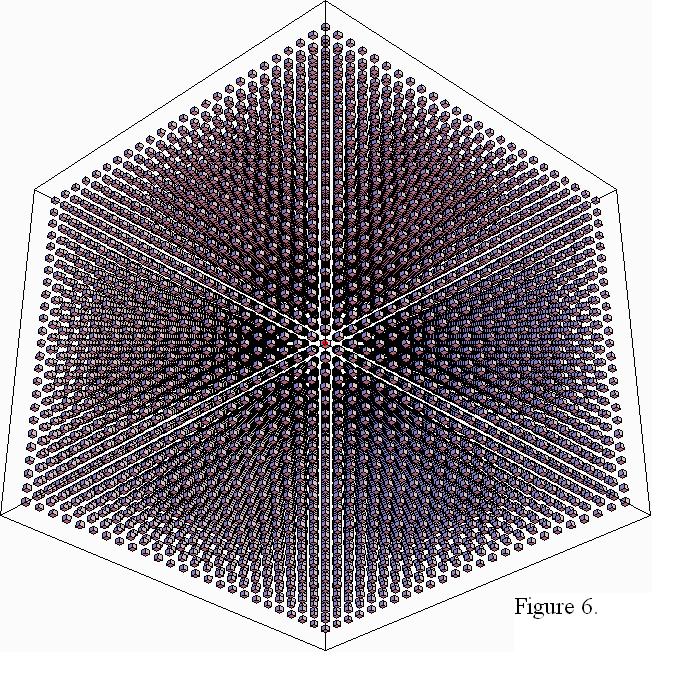}

Figures 4, 5, and 6 above demonstrate the set $\D(C_3, 1/3)$. Figure 4 shows the cubes of that set contained in the region $[0, 10]^3$ and is the same set as Figure 2 but looking out at the origin. Figure 5 shows the cubes inside the region $[0,15]^3$, and Figure 6 shows the cubes contained in $[0, 25]^3$. We know that $\D(C_3, 1/3)$ does not obstruct all views. Below is the corresponding figures with $\D_3=\D(C_3, 1/2)$. Most views are obstructed by the cubes of $\D_3$  contained in the region $[0,10]^3$ as seen in Figure 7. It seems very reasonable that $\D_3$ does obstruct all views. Since the LRC is known to hold for four runners, this is indeed the case.

\includegraphics[scale=.3]{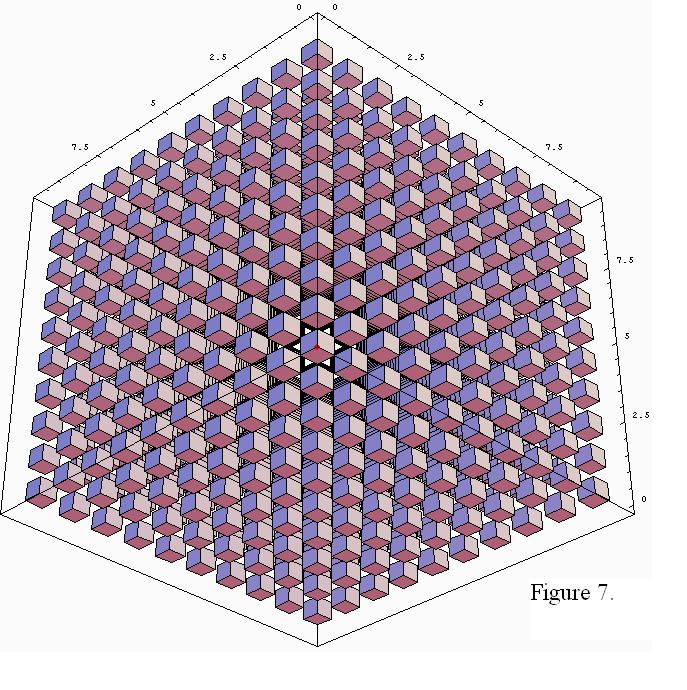}\includegraphics[scale=.3]{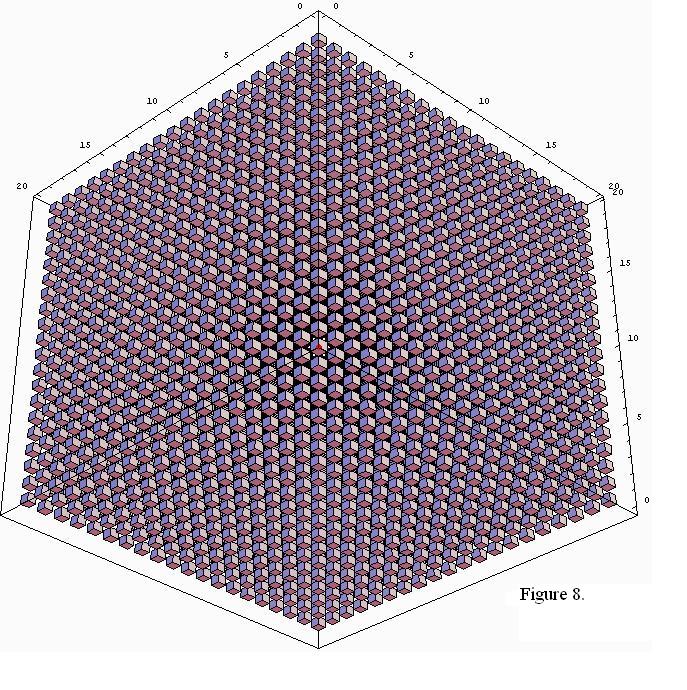}
\includegraphics[scale=.3]{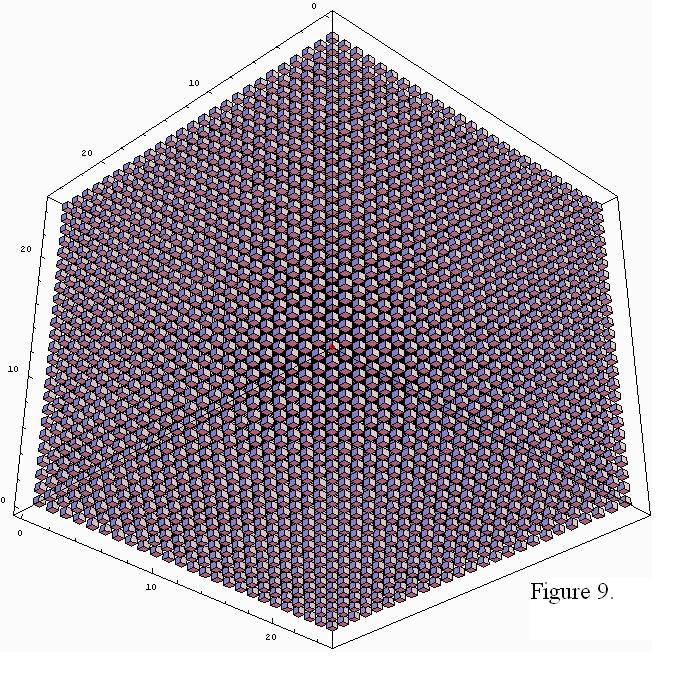}


\begin{example}
Recall Example 1 in the second section. Every orbit of $x=(x_1, x_2) \in \I^2$, has $\ld_2(x) \ge \k(2) = \f{3}$ by definition of $\k$ and Theorem 10. Thus $O(x)$ intersects every square centered at $(\f{2}, \f{2})$ with length greater than $1 - 2\ld_2(x) \le 1 - 2\k(2) =1-\fr{2}{3}= \f{3}$. Since $x$ was arbitrary with nonzero coordinates, it follows that for all $\e > 0$, every orbit of every irrational pair intersects the square with length $\f{3} + \e$ centered at $(\f{2}, \f{2})$. We denote such a square as $G_\e:=G(\f{3} + \e)$.
\end{example}

\begin{example}
Let $s=(s_1, s_2) \in \Q^2$. Then for every $\e > 0,$ the line $y= ts$, $t \in \R$ intersects $G_\e$:
notice that $c=\sup_{t \in \R} \min \set{\norm{t s_1}, \norm{t s_2}} \ge \inf_{s \in \n^2} \del_2(s) = \f{3}$ by definition of $\del_2$, Theorem 10, and since the line $y = ts$ has rational slope. We show $L=\set{ts}_{t \in \R} \con \T^2$ intersects $G_\e$:
By definition of $c$, there are points $(x, y) \in L$ such that $\norm{x} > c -\fr{\e}{2}$ and $\norm{y} > c - \fr{\e}{2}$. Consider the square $G'=G(1-2c + \e)$. If $(x, y)$ did not intersect $G'$ then $\norm{x}$ or $\norm{y}$ would be less than than the distance from a corner of $G'$ to the origin. That is, $\norm{x} < \f{2} - \fr{1-2c+\e}{2} = c - \fr{\e}{2}$ or $\norm{y} < c - \fr{\e}{2}$, which contradicts the choice of that pair. Since $c \ge \f{3}$, we have that $1-2c + \e \le 1-\fr{2}{3} + \e = \f{3} + \e$, so that the square $G_\e$ contains the square $G'$ and hence $L$ intersects $G_\e$. Thus, in the two-torus, every line with rational direction intersects $G_\e$ for any $\e > 0$.
\end{example}

 Examples 2 and 3 perfectly illustrates the relationship between the diophantine problem and the view-obstruction problem for two dimensions, and the results they give can be extended and when applying the results from the previous two sections. Recall Figure 1 in the second section, showing the set of squares with length $1/3$ centered at the half-integers. We denoted this set as $\D_2=\D(C_2, 1/3)$. Reducing the entire plane modulo the integer lattice equates each square in $\D_2$ to the square $G(1/3)$. We know that each ray $r = tx$ with nonzero slope, $t > 0$ is obstructed by $\D_2$ by Theorem 20 in section 3. Thus the reduction of every ray in the two-torus intersects $G(1/3)$. Hence, in the two-torus, every line with rational direction intersects $G(1/3)$, this is stronger than what is shown in Example 3. 


\section{Billiard paths in square Tables}
As one may have noticed, results in the previous sections used various facts about orbits of certain paths in the $n$-torus. Specifically, Theorem 7 in the second section, Proposition 15 and Proposition 5 in the third section all relied on Lemma 6 which used the fact that the orbit $O(x)$ is dense in $\T^n$ when the $x_i$ are rationally independent. We used these results to relate the diophantine and View-obstruction problems. In this section we mainly consider the results in the previous section on view-obstruction and explore its analog to billiard paths in $n$-dimensional cubes. We start by tiling the first quadrant with unit squares. Below, all rays pass through the origin, and all billiard paths start their initial trajectory passing through the origin and has, unless otherwise specified, the unit square as a billiard table.

\begin{proposition}
Each ray in the first quadrant corresponds with a path in a square billiard. Likewise each billiard path in the square corresponds to a ray in the first quadrant.
\end{proposition}

Where a ray ``corresponds" to a billiard path (and vice versa) if the billiard path in the square can be ``unfolded" into a strait line. The following figures demonstrate such an unfolding of the ray $y = \fr{x}{2}$. On the left, Figure 10 shows the ray for $0 \le x \le 4$. In Figure 11 on the right, the corresponding billiard path is shown with each path segment marked with its corresponding ray segment.

\includegraphics[scale = .6]{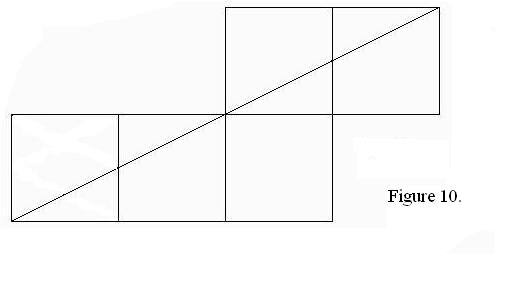} \,\,\,\, \includegraphics[scale = .6]{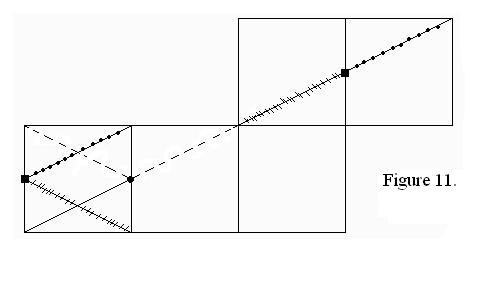}

\begin{proof}[Proof of Proposition]
\emph{(Each billiard path corresponds to a ray)}: To prove the result, we need to associate each billiard path to a ray in the first quadrant. We are assuming that each billiard path has an initial trajectory at the origin. If this first trajectory is vertical or horizontal there is nothing to prove, as each respectively corresponds to a vertical or horizontal ray. We call such billiard paths $trivial$.
Assume then that the initial angle $\phi$ of the trajectory has radian measure $0 < \phi < \pi/2$. By symmetry it suffices to assume that $0< \phi < \pi/4$, for every such trajectory with $\phi > \pi/4$ can be adequately reflected to resemble a path with $\phi < \pi/4$. Every billiard path is entirely characterized by this angle $\phi$, as it uniquely determines the first incident angle of reflection. Instead of reflecting the path inside the billiard table, we reflect the billiard table across the side where this first incident reflects. This process is described in Figure 11 below, where $\theta$ is this first incident angle.

\includegraphics[scale=.7]{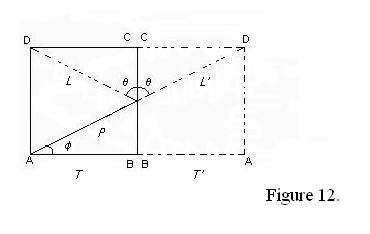}

The subsequent path reflections in the first table $T$ will correspond with their reflection in the new table $T'$. Thus the entire billiard path, (excluding the first segment $P$), corresponds to a mirror image in $T'$. We call this image $I$. The segment $L$ has as its image $L'$. Adjoining $L'$ to $P$ extends the ray segment of angle $\phi$ that had begun with $P$. Repeating this process on the billiards image $I$ in $T'$ proves the demonstration, with one note: If the first segment of the ray does not reflect across a side, but instead hits a corner, then reflect the table about the line of slope $-1$ at this corner and take the images in this new square table to correspond with this reflection. We can thus extend each billiard path to a ray which has the same slope as the initial trajectory of the billiard path.
\\
\emph{(Each ray corresponds to a billiard path)}: If a ray $r$ has slope $\phi > 0$ from the horizontal axis, correspond $r$ with the billiard path that has $\phi$ as its initial trajectory angle. Thus the produced correspondence between the rays and billiard paths is a bijection.
\end{proof}

Given a square billiard table $T$ with unit side lengths, we let $G(\alpha)$ be the square with side length $\alpha$ with the same center as $T$, i.e., $G(\alpha)$ is the scaling of $T$ by $\alpha$. A natural problem is to find the smallest $\alpha$ such that $G(\alpha)$ intersects every nontrivial billiard path. This problem is closely tied to the two-dimensional case of the view-obstruction problem discussed in the last section. Using the following lemma we show that these problems are equivalent.

\begin{lemma}
In a square billiard table, $G(\alpha)$ is invariant under the reflections given in Proposition 21.
\end{lemma}

\begin{proof}
In a square billiard $T$, let $T'$ be its reflection about a side. The statement in the lemma means that the image of $G=G(\alpha)$ in $T'$ corresponds to a square $G'$ with side length alpha with the same center as $T'$. This is immediate from the symmetry of $G$ in $T$ and the fact that the center of $T'$ is the image of the center of $T$, in the reflection. This is shown in Figure 13 below.
\\
\includegraphics[scale = .6]{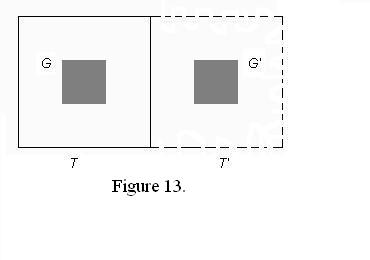}
\end{proof}

\begin{corollary}
Let $T$ be a billiard table and $T'$ a reflection as given in Proposition 20. If $S$ is a billiard path in $T$ that intersects $G=G(\alpha)$, then the image of $S$ in $T'$ intersects $G'$.
\end{corollary}

\begin{proof}
If this were not so then there would be a point inside (resp. on) $G$ that was not reflected inside (on) $G'$, contradicting the lemma.
\end{proof}

Since $T$ is also a reflection of $T'$. A billiard path intersects $G$ if and only if its image in $T'$ intersects $G'$. By inductively applying the above lemma and corollary, the respective results holds for any finite number of reflections of a billiard.

\begin{theorem}
$\inf\set{ \alpha : G(\alpha) \emph{ obstructs every nontrivial billiard path in $T$}} = \f{3}.$
\end{theorem}

\begin{proof}
Let $S$ be a nontrivial billiard path in $T$ with initial an trajectory of $\phi >0$. Let $r$ be the ray corresponding the $S$ according to Proposition 21. Assume for contradiction that $S$ does not intersect $G(1/3)$.  Pick any segment, $J$, of $r$ inside a unit square, $T'$, centered at some half-integer $x$. The construction in Proposition 21 shows that this segment corresponds to a segment of the billiard path $S$ via multiple reflections. By Lemma 21, $G(1/3)$ corresponds by these reflections to a square $G'(1/3)$ that has center $x$, and by Corollary 23, the segment $J$ does not intersect $G'(1/3)$. Thus no segment of the ray $r$ intersects a square with length $1/3$ centered at a half-integer. But this contradicts Theorem 20 in the previous section, for then the set $\D(C_2, \f{3})$ would not obstruct the ray $r$ which has positive slope since $\phi > 0$. Thus, every nontrivial billiard path intersects $G(\f{3})$. The billiard path show in figure 11 only intersects the boundary of $G(\f{3})$. This is shown in Figure 14 below. 

\includegraphics[scale=.5]{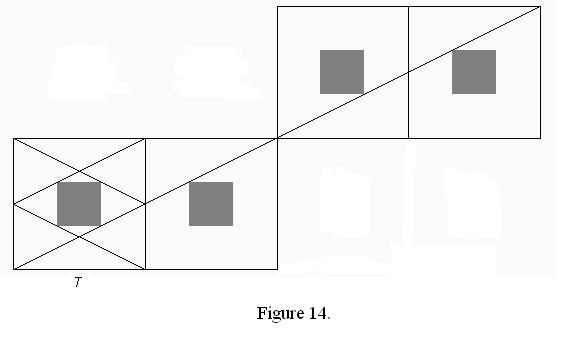}

Thus $1/3$ is the infimum of the $\alpha$ such that $G(\alpha)$ intersects every nontrivial billiard path. 
\end{proof}

\subsection{Billiard Paths in triangular tables}
In this subsection we investigate the same questions covered above but considering billiard paths in a regular triangle, $Q$, of unit side length. As in the previous, all billiard paths have their initial trajectory from the origin, which in our case is the lower left corner of our triangle $Q$. Below is an example of such a billiard with initial trajectory angle of $\phi = \fr{\pi}{4}$.

\includegraphics[scale=.3]{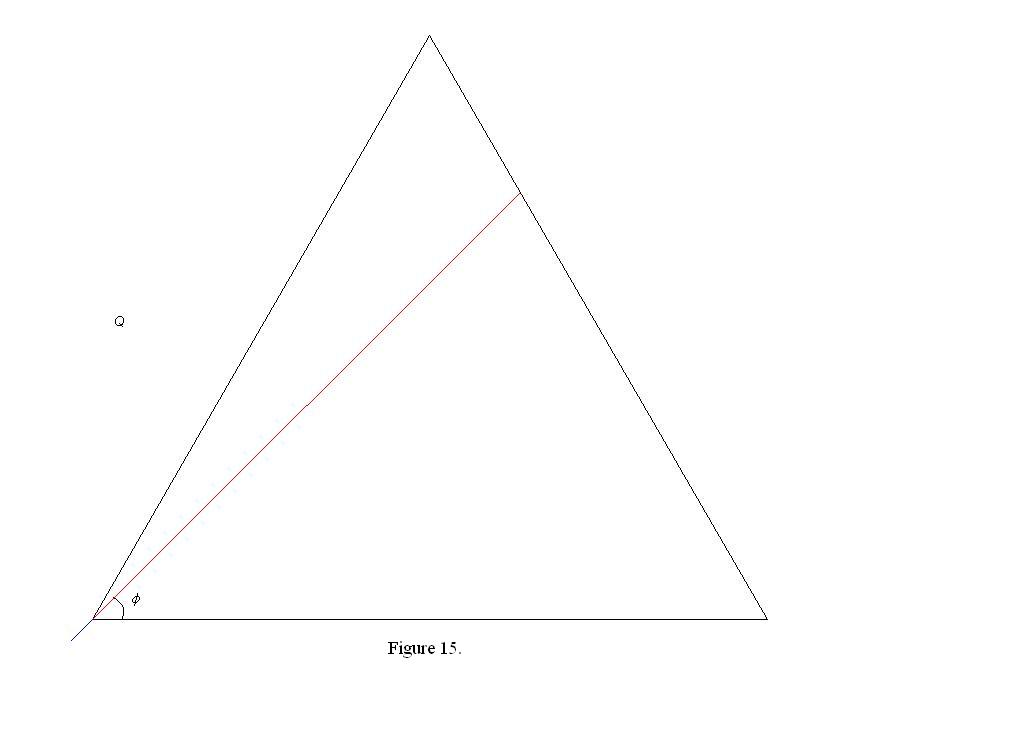} \includegraphics[scale=.3]{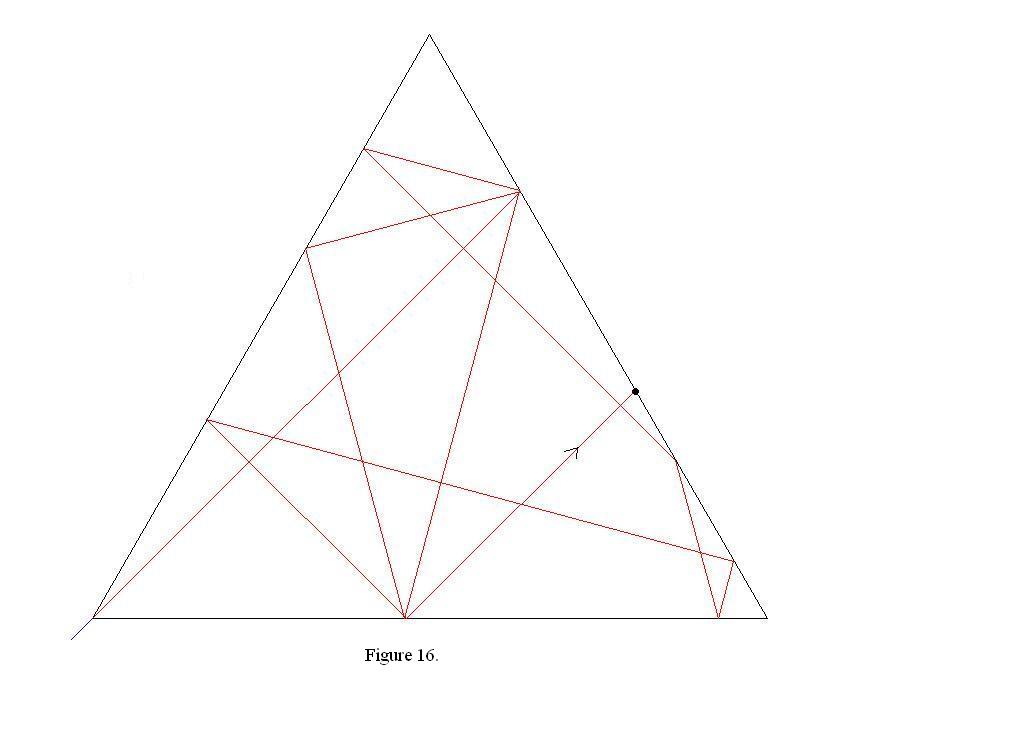}

\includegraphics[scale=.3]{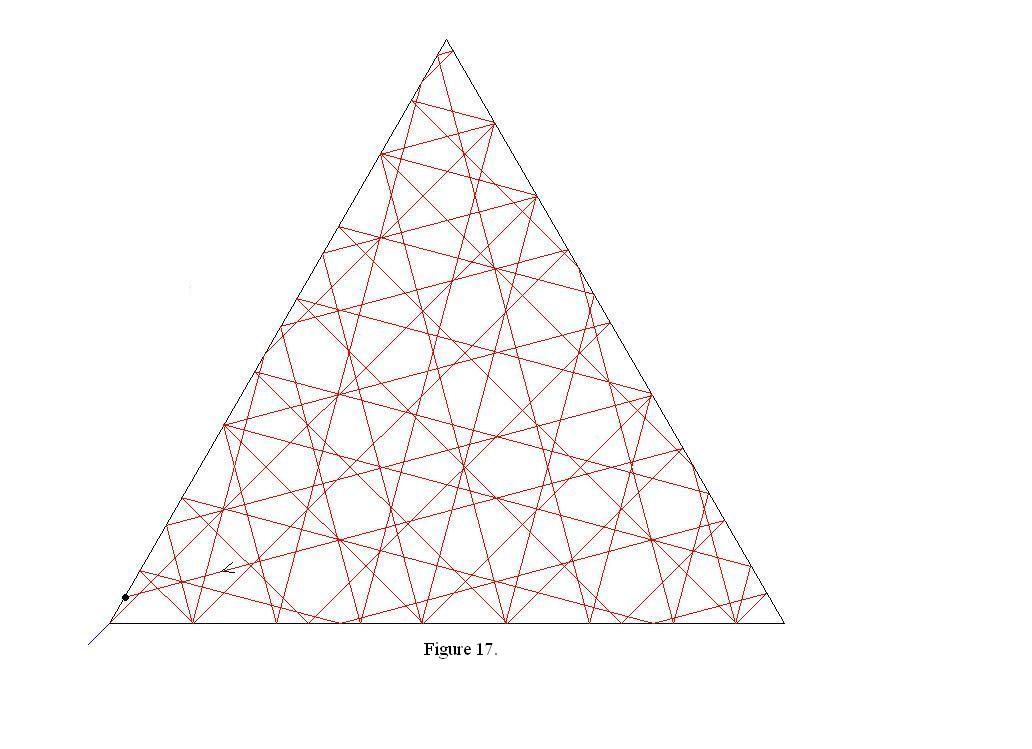} \includegraphics[scale=.3]{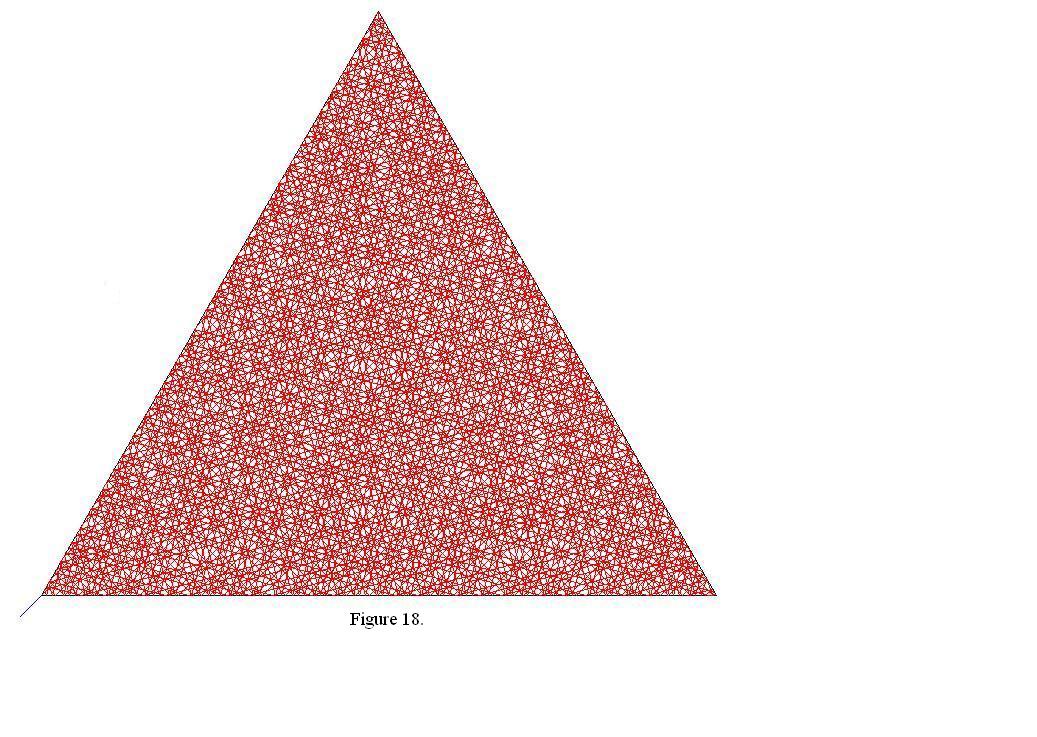}

Figure 16 on the upper right displays ten reflections of the path inside $Q$, with the eleventh table strike taking place at the indicating dot. Figure 17 in the lower left displays 50 reflections with the $51^{\emph{st}}$ strike at the dot. Figure 18 shows the billiard path at five-hundred strikes (499 reflections). For $0 < \alpha < 1$, define $H(\alpha)$ to be the scaling of $Q$ by $\alpha$. Thus $H(\alpha)$ and $Q$ have the same triangular incenter. Figure 19 below displays $H(\alpha)$, for $\alpha = 1/4$, contained in $Q$.

\includegraphics[scale=.8]{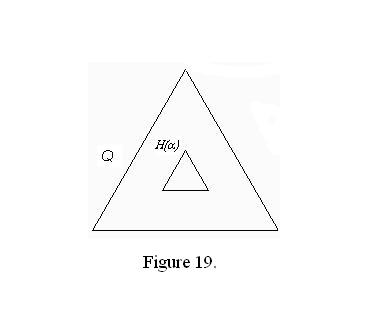}

We solve the analog of Theorem 24 for triangular billiards. That is, we find the smallest number $\alpha$ such that $H(\alpha)$ intersects every nontrivial billiard path in $Q$. Actually, $a \,\, priori$ we cannot know there is a smallest number, but must find 
$$\beta = \inf\set{\alpha : H(\alpha) \text{ intersects every nontrivial billiard path in } Q},$$
and must then check to see if $\beta$ is indeed in the set.
Because the equilateral triangle $Q$ generates a tilling of the euclidean plane with reflections across a side; and since the image of $H(\alpha)$ in the reflected triangle $Q'$ is $H'(\alpha)$, i.e., a scaling of $Q'$ by $\alpha$, we have the analogous results of Proposition 21, Lemma 22, Corollary 23 corresponding to these triangular billiards. That is, instead of rays in the first quadrant, we consider rays the euclidean plane with a horizontal angle between 0 and $\fr{\pi}{3}$.

\begin{proposition}
Each ray with horizontal angle between 0 and $\fr{\pi}{3}$ corresponds with a billiard path in $Q$. Likewise each billiard path in $Q$ corresponds to a ray in horizontal angle between 0 and $\fr{\pi}{3}$.
\end{proposition}

\begin{lemma}
In an equilateral triangular billiard table, $H(\alpha)$ is invariant under the reflections given in Proposition 25.
\end{lemma}

\begin{corollary}
Let $Q$ be an equilateral triangular billiard table and $Q'$ a reflection as in Proposition 24. Then if a billiard path $S$ in $Q$ intersects $H(\alpha)$, then the image of $S$ in $Q'$ intersects $H'(\alpha)$.
\end{corollary}

As is the case with square billiard tables, $Q$ is also a reflection of $Q'$. Thus a billiard path intersects $H(\alpha)$ if and only if its image in $Q'$ intersects $H'(\alpha)$. Inductively applying the above lemma or corollary, the respective results hold for any finite number of reflections, i.e., unfoldings, of a triangular billiard.

We are almost ready to prove the corresponding result of Theorem 24 for triangular billiards. However, we lack the analog of Theorem 20, which says that $\D(C_2, \f{3})$ obstructs all views. Effectively this result says that if one tiles the plane with unit cubes, then scaling each square by $\f{3}$ will produce a set that will block all rays (with positive slope). This was our the main result in view-obstruction that was equivalent to the LRC with three runners. We state the triangular billiards' analog to this as a lemma which is proved after the theorem.

\begin{lemma}
Let $Q$ be an equilateral triangle with unit length, and let unfoldings of $Q$ tile the region between the rays with angles 0 and $\pi/3$. So a sequence of unique triangles $\set{Q_n}_1^\infty$ tiles this region. Then $\set{H_n(\f{4})}_1^\infty$ obstructs all rays in the region, and $\f{4}$ is the least such number to do so.
\end{lemma}

\begin{theorem}
$\inf\set{\alpha : H(\alpha) \text{ intersects every nontrivial billiard path in } Q} = \f{4}$
\end{theorem}
\begin{proof}
Let $Q$ be an equilateral triangle with unit length and $\set{Q_n}_1^\infty$ unfoldings of $Q$ that tile the region described in Lemma 28. Let $S$ be a nontrivial billiard path in $Q$ with initial an trajectory of $0 < \phi < \pi/3$, and let $r$ be the ray corresponding to $S$ according to Proposition 25. Assume for contradiction that $S$ does not intersect $H(1/4)$.  Pick any segment, $J$, of $r$ inside some $Q_n$ with incenter $x$. By Proposition 25 the segment $J$ corresponds to a segment of the billiard path $S$ via multiple unfoldings. By Lemma 26, $H(1/4)$ corresponds by these unfoldings to a regular triangle $H_n(1/4)$ with incenter $x$, and by Corollary 26, the segment $J$ does not intersect $H_n(1/4)$. Thus no segment of the ray $r$ intersects any $H_n(\f{4})$. This immediately contradicts Lemma 28. Thus, every nontrivial billiard path intersects $H(\f{4})$. By Lemma 28, their is a ray that cannot intersect the interiors of the $H_n(\f{4})$. Thus the billiard path corresponding to this ray only intersects the corners of $H(\f{4})$, hence 1/4 is the infimum of the set in the theorem statement.
\end{proof}

We now prove the lemma.
\begin{proof}[Proof of Lemma 28]
We refer to the following figure in the proof:
\\
\includegraphics[scale = .5]{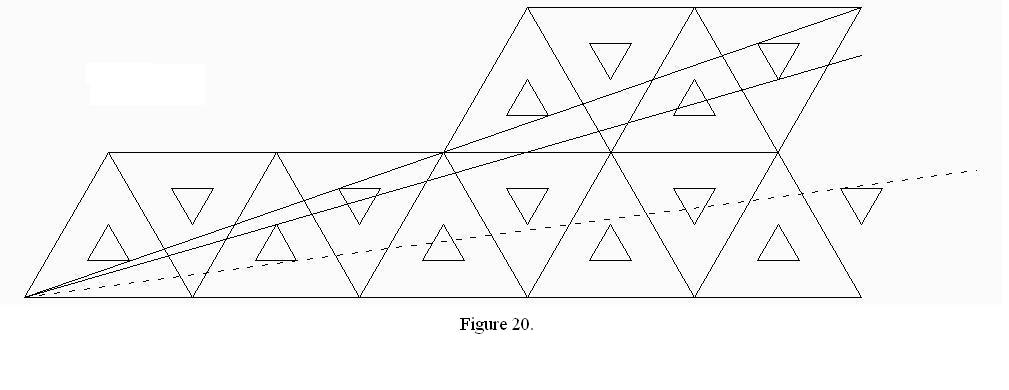}

The smaller triangles represent a portion of the set $W=\set{H_n(\f{4})}_1^\infty$, in a portion of the tiling generated by $Q$. In order to prove that every ray with slope $0< \phi < \fr{\pi}{3}$ is obstructed by the collection $W$, it suffices by symmetry to prove the result for such rays with $0 < \phi < \fr{\pi}{6}.$ 
The upper solid ray has the form $r_1 = \fr{\sqrt{3}\,x}{5}$ for $x > 0$. The ray $r_1$ intersects only the edges of sets in $W$. Call the lower solid ray $r_2$ and the dashed line $r_3$. It is evident that all rays between $r_1$ and $r_2$ intersect $W$, as do all rays between $r_2$ and $r_3$. The gaps between consecutive triangles grow tighter for rays below $r_3$, so that all rays with positive slope under $r_3$ intersect sets in $W$. This proves the lemma.
\end{proof}


\section{Invisible runners and finite fields}
The results of this section originated in 2008 in a paper by the Polish computer scientists Sebastian Czerwinski and Jaroslaw Grytczuk [1]. Recall for $s = (s_1,\dots,s_k) \in \n^k$, we define $\del_k(s) = \ds\sup_{x \in \R} \ds\min_{1 \le i \le k} \norm{x s_i}$ as described in the first section. Proposition 4 says that the LRC is equivalent to $\del_k(s) \ge 1/(k+1)$ for each $s$ and each natural number $k$. We alter this notation slightly, letting $S = \set{s_1, \dots, s_k} \con \n$ be a set of $k$ positive integers, setting $$\del(S) = \ds\sup_{x \in \R} \ds \min_{1 \le i \le k} \norm{x s_i}.$$ It is readily seen that the LRC is equivalent to $\del(S) \ge 1/(k+1)$ for each $k$ element subset of $\n$. We also define $\fl{x}$ to be the usual floor of $x$, and $\set{x}$ to be the fractional part of $x$. We prove two results in this section which are important for several reasons. First, the techniques in this section are more algebraic, leading to an algebraic conjecture that is equivalent to the LRC. Second, if one has a set of runners, these techniques can be used to give a finite algorithm for computing the time a given runner is ``loneliest". This answers a natural question as to whether such an algorithm exists. Let $S = \set{s_1, \dots, s_k} \con \n$, and let $p$ be a prime that does not divide any $s_i$. Thus, the elements of $S$ modulo $p$ is a subset of
$\z_p^*$, which we define as the set of non-zero elements of the field $\z_p=\set{0,1,\dots,p-1}$. We arrange the elements of $\z_p$ on the unit circle in a usual fashion, i.e., as the $p^{\tx{th}}$ roots of unity. For an integer $n$, the image of $n$ in $\z_p$ is denoted by $\cl{n}$.

\begin{lemma}
Let $B = \pm\set{1,2,\dots,m} \con \z_p$ and $S=\set{s_1,\dots, s_k} \con \n$. Suppose there is an $x \in \z_p^*$ such that $\cl{xs_i}$ is not in $B$ for any $i$. Then $\del(S) \ge (m+1)/p$.
\end{lemma}
\begin{proof}
Since $\cl{xs_i} \neq 0$, if $\norm{\cl{xs_i}/p} = \norm{xs_i/p} < \fr{m+1}{p}$ then we must have $\cl{xs_i}/p \in \pm \set{1/p, 2/ p ,\dots, m/p}$ so that $\cl{xs_i} \in \pm B$. This contradicts the hypotheses. Hence $\del(S) = \ds\sup_{x \in \R} \ds\min_{1\le i \le k} \norm{xs_i} \ge (m+1)/p$. If the above holds for $m = \fl{p/(k+1)}$ then $\del(S)
\ge \ds\min_{1 \le i \le k} \norm{ts_i} \ge 1/(k+1)$ where $t = x/p$.
\end{proof}

\begin{proposition}
Let $S = \set{s_1, \dots,s_k}$ be a set of positive integers. Let $\e >0$ and let $p>\fr{k}{\e} + 1$ be a prime number that does not divide any element in $S$. Then for every $d \in \set{0, 1,\dots,k}$ and
$B \con \z_p^*$ with $|B| \le p(d+1)/(k+\e)$, there is an $x \in \z_p^*$ such that $|B \cap xS| \le d$.
\end{proposition}
\begin{proof}
Consider a rectangular $k \times (p-1)$ matrix $A = (a_{ij})$ defined by $a_{ij} = \cl{j s_i}$. We need to show that there is a column in $A$ with at most $d$ entries belonging to $B$:
Let $T$ be the total number of positions in $A$ occupied by elements of $B$. Since $\z_p$ is a field and each $s_i$ is nonzero, every row of $A$ consists of the whole of $\z_p^*$. Thus $T=k|B|$, and the hypothesis on $|B|$ implies that $T \le kp(d+1)/(k+\e)$. If every column in $A$ had at least $d+1$ entries in $B$, then $T \ge (p-1)(d+1)$. Hence,

$$k\fr{p(d+1)}{k+\e} \ge (p-1)(d+1),$$
and so 
$$\fr{k}{k+\e} \ge \fr{(p-1)}{p}.$$
But by assumption $p > \fr{k}{\e} +1$ so that $(p-1) > \fr{k}{\e}$ so $\fr{(p-1)}{p} > \fr{k}{p\e}$. Also $p > \fr{k}{\e}+1$ gives $p\e > k + \e$, so that 
$\fr{(p-1)}{p} > \fr{k}{p\e} \ge \fr{k}{k+\e}$. This is a contradiction. Hence at least one column has no more than $d$ entries of elements in $B$.
\end{proof}

We use the above results to prove the first important theorem of this section.

\begin{theorem}
Let $k$ and $d$ be arbitrary integers, $0 \le d < k.$ Then every set $S$ of k positive integers contains a subset $D$ of size $k-d$ such that $\del(D) \ge (d+1)/(2k)$.
\end{theorem}

In the special case of $d = 1$, this says that if we are given a set of integers of size $k$, there is a subset of size $k-1$ with $\del(D) \ge (d+1)/(2k) = 1/k$. In light of the first section, this says if we are given a set of $k+1$ runners, removing a certain runner (or making him ``invisible") will give us $k$ runners where the runner with speed 0 becomes lonely in the sense of $k+1$ runners. If $k\ge 6$ (so that $\fr{3}{2k} \ge \f{k-2}$), then the case $d=2$ says that given $k+2$ runners, removing 2 runners will give a remaining set of $k$ runners $D$ with $\del(D) \ge 3/(2k) \ge 1/(k-2)$. That is, every set of $k+2$ runners contains a set of $k$ runners where the runner with constant 0 speed becomes lonely $\emf{regardless of the size of k}$, provided $k \ge 6$.

\begin{proof}[Proof of theorem 32]
Let $0 \le d < k$ be fixed and let $S$ be any set of $k$ positive integers. Let $\e_n > 0$ with $\e_n \to 0$ as $n \to \infty$. For every $n$, let $p_n$ be a prime such that $p_n > \fr{k}{\e_n} + 1$.
Set $m_n = \fl{p_n(d+1)/(2(k+\e_n))}$ and $B_n = \pm\set{1,2,\dots,m_n}$. By Proposition 31 there is an $x_n \in \z_{p_n}^*$ with $|B_n \cap x_nS| \le d$. Let $D_n = \set{s \in S: x_ns \notin B_n}$. Now since $|B_n \cap x_nS| \le d$, we have that $|D_n| \ge k - d$ for each $n \ge 1$. By the lemma, it follows that 

$$\del(D_n) \ge \fr{m_n + 1}{p_n} \ge \fr{p_n(d+1)/(2(k+ \e_n)}{p_n}=\fr{d+1}{2(k+ \e_n)}.$$
Since $S$ is a finite set, there are infinitely many $n$ for which $D_n \con S$ is the same subset. Call this subset $D$. Since $\e_n \to 0$. We have that 

$$\del(D) \ge \ds\lim_{n \to \infty} \fr{d+1}{2(k + \e_n)} = \fr{d+1}{2k}.$$ This proves the theorem.
\end{proof}

\begin{proposition}
If $S = \set{s_1, s_2,\dots,s_k} \con \n$, then $\del(S)$ is attained for $x_0 = a/(s_i + s_j)$ for some $i \neq j$ and some positive integer $a < s_i + s_j + 1$. 
\end{proposition}

\begin{proof}
Define $f_S(x) = \ds\min_{1\le i \le k} \norm{x s_i}$ for $x \in \T$. By continuity of $f_S$ and the fact that $\T$ is compact, there is an $x_0 \in \T$ where $f_S$ attains its maximum. Thus $\del(S)=\ds\sup_{x \in [0,1]} \ds\min_{1 \le i \le k} \norm{x s_i}= f_S(x_0)$, and let $s_i \in S$ be an integer for which $\del(S) = \norm{x_0 s_i}$. We show that there must be another $j \neq i$ such that $\norm{x_0s_i} = \norm{x_0 s_j}$. If there was not any such $j$, then $\ds\min_{1\le i\le k} \norm{x_0s_i}=\norm{x_0s_i} < \norm{x_0s_j}$ for $j \neq i$. Thus choosing $\e > 0$ small enough we will have $\ds\min_{j \neq i} \norm{(x_0 \pm \e)s_j} > \norm{(x_0 \pm \e)s_i}$ by continuity. But either $\norm{(x_0+\e)s_i} > \norm{x_0s_i}$, or $\norm{(x_0-\e)s_i}>\norm{x_0s_i}$, so that $f_S(x_0)$ is not the really the maximum of $f_S$. Similarly we can show there is such a $j$ with $s_j \not s_i$. Since $\norm{x_0s_i} = \norm{x_0s_j}$, we must have $\set{x_0s_i} = 1 - \set{x_0s_j}$. Hence 
$$x_0s_i+x_0s_j = \fl{x_0s_i} + \fl{x_0s_j} + \set{x_0s_i} + \set{x_0s_j} = \fl{x_0s_i} + \fl{x_0s_j}+1 := a,$$ which results in $x_0 = a/(s_i+s_j)$ satisfying the required properties.
\end{proof}

Applying Proposition 33, we have the following equivalence to the LRC,

\begin{conj}
For every set $S \con \n$ of size $k$, there is a natural number $n$, and $x \in \z_n$, such that $xS \cap B = \emptyset$ for $B = \pm\set{0,1,\dots, m}$ where $m = \ceil{n/(k+1)}-1$.
\end{conj}

\begin{proof}[Proof of Equivalence]
If the LRC is true then for any $k$ element set $S \con \n$ we have $\del(S) \ge 1/(k+1)$, so by Proposition 33, the above conjecture will hold with $n = s_i + s_j$. This follows since, by the proposition, $\norm{\fr{a}{n}s_l} \ge \del(S) \ge 1/(k+1)$ so that the numbers $\cl{as_l}=as_l\!\!\mod(n)$ fall outside of the set $B= \pm \set{0,1,\dots,m}$ with $m= \ceil{n/(k+1)}-1$. Otherwise if $\cl{as_l} \in B$ then $\norm{as_l/n}=\norm{\cl{as_l}/n} = \norm{q/n}$ for some $q \in B$. But all such $q \in \pm \set{0,1,\dots, m}$ have $q/n \in \pm \set{0, 1/n,\dots, m/n}$ and since $m/n \le (\ceil{\fr{n}{k+1}}-1)/n < \fr{n}{k+1}/n = 1/(k+1)$, we have $\norm{q/n} \le \norm{\pm m/n} < 1/(k+1)$ which would be a contradiction. This same argument shows that if there is an $n$ and $x \in \z_n$ where $xS \cap B = \emptyset$, we have $\norm{xs_i} \ge 1/(k+1)$ for all $1 \le i \le k$, which implies $\del(S) \ge 1/(k+1)$. If this happens for all $S \con \n$ of arbitrary size $k$, then the LRC holds by Proposition 4.
\end{proof}

Given a set of runners with positive integer speeds $\set{s_i}_1^k$, Proposition 33 shows that we can compute, in finite steps, the time a certain runner with speed $s_i$ becomes lonely. Let $r_j = s_j - s_i$ for all $j \neq i$. Set $M = \ds\max_{l \neq k} (r_j + r_k)$. Then compute $a_{(j, l)} = j+l$ for all $a_{(j,l)} \le M+1$, which is a finite computation. By Proposition 33, one of the numbers in the set 
$$\set{\fr{a_{(j,l)}}{(r_q + r_m)}}_{j,l,q,m =1}^{k-1}$$ 
is a time when $s_i$ becomes lonely. Or, more precisely, since we cannot assume $s_i$ is ever lonely, the set above contains a time when $s_i$ is the furthest distance possible from every other runner.

\end{document}